\documentclass[11pt,reqno]{amsart}
\usepackage{amsmath, amsfonts, amsthm, amssymb}

\theoremstyle{plain}
\newtheorem{lemma}{Lemma}[section]
\newtheorem{theorem}{Theorem}[section]

\theoremstyle{definition}
\newtheorem{definition}{Definition}[section]
\theoremstyle{remark}
\newtheorem{remark}{Remark}[section]

\numberwithin{equation}{section}

\newcommand{\p}{\partial}
\newcommand{\norm}[1]{\left\Vert#1\right\Vert}

\newcommand{\dist}{\mathrm{dist}}
\newcommand{\diam}{\mathrm{diam}}

\newcommand{\R}{\mathbb{R}}

\newcommand{\fei}{\varphi}

\makeatletter

\newcommand{\Rmnum}[1]{\mathrm{\expandafter\@slowromancap\romannumeral#1@}}
\makeatother

\begin{document}
\title[Uniqueness of Transonic Shock Solutions]
{Uniqueness of Transonic Shock Solutions in a Duct for Steady
 Potential Flow}

\author{Gui-Qiang Chen}
\author{Hairong Yuan}
\address{Gui-Qiang Chen:  School of Mathematical Sciences, Fudan University,
 Shanghai 200433, China; Department of Mathematics, Northwestern University,
 Evanston, IL 60208-2730, USA}
\email{gqchen@math.northwestern.edu}
\address{Hairong Yuan:
Department of Mathematics, East China Normal University, Shanghai
200241, China} \email{hryuan@math.ecnu.edu.cn,\,
hairongyuan0110@gmail.com}

\keywords{uniqueness, transonic shock, free boundary, Bernoulli law,
maximum principle, potential flow, duct}
\subjclass[2000]{35J25,35B35,35B50,76N10,76H05}
\date{\today}

\begin{abstract}
We study the uniqueness of solutions with a transonic shock in a
duct in a class of transonic shock solutions, which are not
necessarily small perturbations of the background solution, for
steady potential flow. We prove that, for given uniform supersonic
upstream flow in a straight duct, there exists a unique uniform
pressure at the exit of the duct such that a transonic shock
solution exists in the duct, which is unique modulo translation. For
any other given uniform pressure at the exit, there exists no
transonic shock solution in the duct. This is equivalent to
establishing a uniqueness theorem for a free boundary problem of a
partial differential equation of second order in a bounded or
unbounded duct. The proof is based on the maximum/comparison
principle and a judicious choice of special transonic shock
solutions as a comparison solution.
\end{abstract}
\maketitle


\section{Introduction and Main Results}

In the recent years, there has been an increasing interest in the
study of transonic shock solutions in ducts or nozzles for the
steady potential flow equation or steady full Euler system for
compressible fluids.
The basic strategy is to construct first some transonic shock
solutions and then study the stability of these solutions by
perturbations of the boundary conditions; see
\cite{ChF1,ChF2,ChF3,CCF,CY,XY,Yu2,Yu3} and the references cited
therein. On the other hand, some basic properties of these special
transonic shock solutions, such as the uniqueness in a large class
of solutions, have not been fully understood. As a first step, in
this paper, we study the uniqueness of solutions with a flat
transonic shock in a straight duct in a class of transonic shock
solutions, which are not necessarily small perturbations of the
background solution, for steady potential flows.
Some classical, related results on transonic flows may be found in
\cite{CC,Morawetz1} and the references cited therein.

Consider steady isentropic irrotational inviscid flows in a finite
duct $D:=(-1,1)\times\Omega\subset\R^3$ or a semi-infinitely long
duct $D':=(-1,\infty)\times\Omega\subset\R^3$, where $\Omega\subset
\R^2$ is a bounded domain with $C^3$ boundary.

The governing equations of potential flows are the conservation of
mass and the Bernoulli law
(cf. \cite{CF}):
\begin{eqnarray}
&&\nabla\cdot(\rho\nabla\fei)=0,\label{101}\\
&&\frac{1}{2}|\nabla\fei|^2+ i(\rho)=b_0, \label{102}
\end{eqnarray}
where $\varphi$ is the velocity potential (i.e., $\nabla\varphi$ is
the velocity), $b_0$ is the Bernoulli constant determined by the
incoming flow and/or boundary conditions, $\rho$ is the density, and
$$
i'(\rho)=\frac{p'(\rho)}{\rho}=\frac{c^2(\rho)}{\rho}
$$
with $c(\rho)$ being the sound speed and $p(\rho)$ the pressure. For
polytropic gas, by scaling,
\begin{equation}\label{gamma-law}
p(\rho)=\frac{\rho^\gamma}{\gamma},\qquad c^2(\rho)=\rho^{\gamma-1},
\qquad i(\rho)=\frac{\rho^{\gamma-1}-1}{\gamma-1}, \qquad \gamma>1.
\end{equation}
%
%
%
%
In particular, when $\gamma=1$ as the limiting case $\gamma\to 1$,
\begin{equation}\label{gamma-1}
i(\rho)=\ln \rho.
\end{equation}
Expressing $\rho$ in terms of $|\nabla\varphi|^2$:
$$
\rho=\rho(|\nabla\varphi|^2)
=\Big(1+(\gamma-1)(b_0-\frac{1}{2}|\nabla\varphi|^2)\Big)^{\frac{1}{\gamma-1}}
\qquad \text{for}\,\, \gamma>1,
$$
or
$$
\rho=\rho(\nabla\varphi|^2) =e^{-\frac{1}{2}|\nabla\varphi|^2+b_0}
\qquad \text{for}\,\, \gamma=1,
$$
equation \eqref{101} becomes
\begin{equation}\label{103-a}
\nabla\cdot(\rho(|\nabla\fei|^2)\nabla\fei)=0.
\end{equation}
Equation \eqref{103-a} is a second order equation of mixed
elliptic-hyperbolic type for $\fei$ in general; it is elliptic if
and only if the flow is subsonic, i.e., $|\nabla\fei|<c$ or
equivalently, $|\nabla\fei|<
c_*:=\sqrt{\frac{2}{\gamma+1}\big(1+(\gamma-1)b_0\big)}$ for
$\gamma>1$ and $c_*=1$ for $\gamma=1$.

\medskip
We first consider the case of a finite duct $D$. Let
$\Gamma=[-1,1]\times\p \Omega$ be the lateral wall, and let
$\Sigma_i=\{i\}\times \Omega$, $i=-1,1$, be respectively the entry
and exit of $D$. That is, we assume

\medskip
\noindent {\rm ($H_1$)} \qquad\qquad\qquad $\p_0\fei\ge0$
\qquad on \,\,\, $\Sigma_i=\{i\}\times\Omega,\,\, i=-1,1$.

\medskip
\noindent Note that $\p D=\Sigma_{-1}\cup\Sigma_{1}\cup\Gamma$. We
are interested in the case that the flow is uniform and supersonic
(i.e., $|\nabla\fei|>c$) on $\Sigma_{-1}$; subsonic on $\Sigma_1$
with uniform pressure. More specifically, for a constant $u^-\in(
c_*, \sqrt{2(b_0+\frac{1}{\gamma-1})})$ and a constant $c_1\in(0,
c_*)$,
we consider the following
problem:
\begin{eqnarray}
& \mbox{\eqref{103-a}}
                   &\text{in}\ \  D,\label{103}\\
& \fei=-u^-, \quad  \p_{x_1}\fei=u^-   &\text{on}\ \  \Sigma_{-1},\label{104}\\
&|\nabla\fei|=c_1  &\text{on}\ \  \Sigma_{1}, \label{106}\\
&\nabla\fei\cdot n=0 & \text{on}\ \  \Gamma, \label{0107}
\end{eqnarray}
where $n$ is the outward unit normal on $\Gamma$.

We remark that the formulation of this boundary problem is
physically natural. Since the flow is supersonic near $\Sigma_{-1}$,
i.e., the equation is hyperbolic on $\Sigma_{-1}$, there should be
initial data like \eqref{104} due to $(H_1)$ (Our choice of $\fei$
in \eqref{104} makes the solution of the  uniform supersonic
upstream flow in $D$ looks neatly; see Lemma \ref{lem101} below). On
the other hand, since the equation is elliptic on $\Sigma_1$, only
one boundary condition is necessary. We choose the Bernoulli-type
condition \eqref{106} since, from the physical point of view,
assigning the pressure (i.e. density for isentropic flow) is of more
interest (cf. \cite{CF}), which is just a boundary condition like
\eqref{106} due to \eqref{102}. Condition \eqref{0107} is the
natural impermeability condition, i.e., the slip boundary condition,
on the lateral wall for inviscid flow.

We are interested in the class of piecewise smooth solutions with a
transonic shock for problem \eqref{103}--\eqref{0107}.

\begin{definition}\label{def101}
For a $C^1$ function $x_1=f(x_2,x_3)$ defined on $\bar{\Omega}$, let
\begin{eqnarray*}
&& S=\{(f(x_2,x_3), x_2, x_3)\in D \,:\, (x_2,x_3)\in \Omega\},\\
&&D^-=\{(x_1,x_2,x_3)\in D\, :\, x_1<f(x_2,x_3)\},\\
&&D^+=\{(x_1,x_2,x_3)\in D\,:\, x_1>f(x_2,x_3)\}.
\end{eqnarray*}
Then $\fei\in C^{0,1}(D)\cap C^2(D^-\cup D^+)$ is a {\it transonic
shock solution} of \eqref{103}--\eqref{0107} if it is supersonic in
$D^-$ and subsonic in $D^+$, satisfies equation \eqref{103-a} in
$D^-\cup D^+$ and the boundary conditions \eqref{104}--\eqref{0107}
pointwise, the Rankine-Hugoniot jump condition:
\begin{equation}
\rho(|\nabla\fei^+|^2)\nabla\fei^+\cdot
\nu=\rho(|\nabla\fei^-|^2)\nabla\fei^-\cdot \nu \qquad\text{on}\,\,
S,\label{rh}
\end{equation}
and the physical entropy condition:
\begin{equation}
\rho(|\nabla\fei^+|^2)>\rho(|\nabla\fei^-|^2) \quad
\Longleftrightarrow \quad |\nabla\fei^+|<|\nabla\fei^-|
\qquad\text{on}\,\, S, \label{entropy}
\end{equation}
where $\nu$ is the normal vector of $S$, and $\fei^+$ ($\fei^-$) is
the right (left)  limit of $\fei$ along $S$.  The surface $S$ is
also called a {\it shock-front}.
\end{definition}

\begin{remark}
Note that, across a transonic shock-front $S$, the potential $\fei$
is continuous, while the velocity $\nabla\fei$ is discontinuous.
Since the shock-front is a free boundary that requires to be solved
simultaneously with the flow behind it, we have to deal with a free
boundary problem indeed. In the following, we also write
$\fei^\pm=\fei|_{D^{\pm}}$ with $\fei^-$ the supersonic flow and
$\fei^+$ the subsonic flow.
\end{remark}

We first state two direct facts.

\begin{lemma}\label{lem101}
There exists a unique supersonic flow solution $\fei^-$ that
satisfies \eqref{103}--\eqref{104} and \eqref{0107} in the class of
$C^2$ supersonic flow solutions in $D$. The unique solution is
$\fei^-=u^-x_1$.
\end{lemma}

\begin{proof}
It is clearly that $\fei^-=u^-x_1$ solves \eqref{103}--\eqref{104}
and \eqref{0107}. By standard energy estimates for hyperbolic
equations, this solution is unique in the class of $C^2$ supersonic
flow.
\end{proof}

\begin{lemma}\label{lem103}
For each $t\in(-1,1)$, the function with a flat transonic
shock-front:
\begin{eqnarray}
\fei_t(x_1,x_2,x_3)=\begin{cases}
u^-x_1,& -1\le x_1<t,\\
u^+(x_1-t)+u^-t, & t<x_1,
\end{cases}
\end{eqnarray}
solves problem \eqref{103}--\eqref{0107} with $c_1=u^+$, where
$u^+\in(0, c_*)$ is determined by $u^-$, $b_0$, and $\gamma\ge 1$.
\end{lemma}

\begin{proof}
This is equivalent to solving $u^+$ from the following two algebraic
equations deduced from \eqref{102} and \eqref{rh}:
\begin{eqnarray}
&&\rho^-u^-=\rho^+u^+,\\
&&\frac{(\rho^+)^{\gamma-1}-1}{\gamma-1}+\frac{1}{2}(u^+)^2
=\frac{(\rho^-)^{\gamma-1}-1}{\gamma-1}+\frac{1}{2}(u^-)^2.
\end{eqnarray}
The calculation similar to Proposition 3 in \cite{Yu1} indicates
that there exists a unique solution $u^+<c_*<u^-$. One can then
easily verify that $\fei_t$ constructed above is a transonic shock
solution.
\end{proof}

We remark that this special solution has played a significant role
in the study of transonic shocks in the recent years for the
potential flow equation and the full Euler system (cf.
\cite{ChF1,ChF2,ChF3,CCF,CY,XY,Yu2,Yu3}). It has been observed to be
unstable if the pressure is given  at the exit in general. Theorem
\ref{thm101} below provides a simple and direct confirmation of this
instability.

\begin{theorem}\label{thm101}
Under assumption {\rm ($H_1$)}, for given $u^-\in(c_*,
\sqrt{2(b_0+\frac{1}{\gamma-1})})$, problem
\eqref{103}--\eqref{0107} is solvable for transonic shock solutions
in the sense of Definition {\rm 1.1} if and only if $c_1=u^+$, with
$u^+$ being a  constant in $(0, c_*)$ determined by $u^-$, $b_0$,
and $\gamma\ge 1$. In addition, the solution is unique modulo
translation: it is exactly $\fei_t$ for $t\in(-1,1)$.
\end{theorem}

This result especially implies that, if the pressure is posed at the
exit, then the boundary value problem is ill-posed in most cases,
and the special transonic shock solutions $\fei_t$ that are widely
studied are not physically stable (cf. \cite{CY,Yu1,Yu2}).  We
remark that, unlike the previous works
\cite{ChF1,ChF2,ChF3,CCF,CY,XY,Yu2,Yu3} where the stability of the
special solutions was studied under small perturbations of the
upstream supersonic flow \eqref{104} or the shape of the wall
$\Gamma$ of the duct, our results do not require such small
perturbations. Our proof is global and based on the
maximum/comparison principle and a judicious choice of special
transonic shock solutions as a comparison solution. It reveals the
basic uniqueness property of such special transonic shock solutions.

\medskip
Next, we focus on the uniqueness of transonic shock solutions in
semi-infinite duct $D'=(-1,\infty)\times \Omega$, with the following
assumption only for the case $\gamma=1$ that

\medskip
\noindent {\rm ($H_2$)}\qquad there exists $k_0>0$ such that
$|\nabla\fei|$ is bounded in $(k_0,\infty)\times \Omega$.
\medskip
%

We remark that this assumption is automatically satisfied for
potential flow with $\gamma>1$, since the velocity should be less
than the critical value $\sqrt{2(b_0+\frac{1}{\gamma-1})}$ due to
the Bernoulli law and the fact that the constant $b_0$ has been
fixed by the supersonic data at the entry.

Let $\Gamma=(-1,\infty)\times\p\Omega.$  We have the following
result:

\begin{theorem}\label{thm102}
Consider problem \eqref{103}--\eqref{104} and \eqref{0107} with $D$
replaced by $D'$ 
and \eqref{106} replaced by {\rm $(H_2)$}. Assume that, for the
solution $\fei$, $|D^2\fei|$ is also bounded in $(k_0,\infty)\times
\Omega$. Then this problem is solvable for transonic shock
solutions, and the solution is unique modulo translation.
\end{theorem}

\begin{remark} This result indicates that, for transonic shock solutions in
semi-infinitely long ducts, there should be no additional asymptotic
condition such as
\begin{eqnarray}\label{116}
\lim_{x_1\rightarrow\infty}\max_{(x_2,x_3)\in\Omega}
|\nabla\fei(x_1,x_2,x_3)|=c_1.
\end{eqnarray}
Otherwise, it is either overdetermined or superfluous. We just need
the reasonable assumptions that the velocity and acceleration of the
flow are bounded. This indicates that the apriori assumptions on the
asymptotic behavior of transonic shock solutions in an infinitely
long duct in \cite{ChF2,ChF3,CCF} may not be necessary.
\end{remark}

\begin{theorem}\label{rem3} In Theorem {\rm \ref{thm102}}, if we
replace {\rm $(H_2)$} by the following stronger assumption:

\medskip
\noindent {\rm $(H_2')$}\qquad \qquad $|\nabla\fei|<\tilde{c}<c_*$\
\qquad in \ $(k_0,\infty)\times \Omega$ \, \, with $\tilde{c}$ a
constant,
\medskip

\noindent that is, the far-away-flow field behind the shock-front is
always subsonic, then the requirement that  $|D^2\fei|$ is bounded
can be removed.
\end{theorem}

In the rest of this paper,  Sections 2--4, we establish Theorems
\ref{thm101}--\ref{rem3}, respectively.

\section{Proof of Theorem \ref{thm101}}

We divide the proof into three steps.

\smallskip
{\it Step 1}. Let $\fei$  be a transonic shock solution of problem
\eqref{103}--\eqref{0107}, and let $S$ be the corresponding
shock-front with equation $x_1=f(x_2,x_3)$, and
$\tau=\min_{\bar{\Omega}}\{f(x_2,x_3)\}$. Then
\begin{eqnarray}\label{add1}
\fei=u^-f(x_2,x_3) \qquad\text{on}\,\, S,
\end{eqnarray}
since the potential $\fei$ is continuous across the shock-front.

\medskip
{\it Step 2}. We notice that \eqref{add1} and the Neumann condition
\eqref{0107} imply that $S$ is perpendicular to $\Gamma$. In fact,
for $P\in\Gamma\cap\overline{S}$, let the normal vector of $\Gamma$
at $P$ be $n=(0,n_2,n_3)$, and denote the normal vector of $S$ at
$P$ to be $\nu=(1, -\p_{x_2} f, -\p_{x_3}f)$. We now show that
$n\cdot\nu=0$.

By \eqref{add1}, we obtain that, at $P$,
\begin{eqnarray}
\p_{x_i}\fei+\p_{x_1}\fei\p_{x_i}f=u^-\p_{x_i}f \qquad\mbox{for
$i=2,3$.}
\end{eqnarray}
Therefore, by \eqref{0107},
\begin{eqnarray}\label{add2}
(u^--\p_{x_1}\fei)\sum_{i=2}^3n_i\p_{x_i}f
=\sum_{i=2}^3n_i\p_{x_i}\fei_i=0.
\end{eqnarray}
By \eqref{entropy}, after passing $S$, we have $\p_{x_1}\fei<u^-$.
Therefore, $\sum_{i=2}^3n_i\p_{x_i}f =0$ and $n\cdot\nu=0$.

\smallskip
{\it Step 3}. Now let  $\fei_a:=\fei_{\tau}$ be the transonic shock
solution constructed in Lemma \ref{lem103},  $D^+$ and $D_a^+$ be
the corresponding subsonic region of $\fei$ and $\fei_a$, and
$D^*:=D_a^+\cap D^+=D^+$. Then $\psi=\fei_a-\fei$ satisfies the
linear, uniformly elliptic equation:
\begin{eqnarray}
&&L\fei=\sum_{i,j}a_{ij}(x)\p_{x_ix_j}\psi + \sum_{i} b_i(x)\p_{x_i}\psi\nonumber\\
&&\quad\ \ \, :=\rho(|\nabla\fei_a|^2)\Delta\psi
-2\sum_{i,j}\rho'(|\nabla\fei_a|^2)\p_{x_i}\fei_a\p_{x_j}\fei_a\p_{x_ix_j}\psi\nonumber\\
&&\qquad\qquad+\Delta\fei\big(\rho(|\nabla\fei_a|^2)-\rho(|\nabla\fei|^2)\big)\nonumber\\
&&\quad\quad\ \
-2\sum_{i,j}\p_{x_ix_j}\fei\Big(\rho'(|\nabla\fei_a|^2)\p_{x_i}\fei_a\p_{x_j}\fei_a
-\rho'(|\nabla\fei|^2)\p_{x_i}\fei\p_{x_j}\fei\Big)\nonumber\\
&&\quad\ \ \, =0 \qquad\qquad \text{in}\ \ D^*. \label{112}
\end{eqnarray}
Since $D$ is bounded, by our assumption $\varphi\in C^2$,
$|\nabla^2\varphi|$ is bounded in $D^+$.

The boundary conditions are
\begin{eqnarray}
&\nabla(\fei_a+\fei)\cdot\nabla\psi=(u^+)^2-c_1^2 \qquad
&\text{on}\,\, \Sigma_1,
\label{1160}\\
&\nabla\psi\cdot n=0 \qquad &\text{on}\,\, \Gamma\cap\overline{D^*}.
\label{117}
\end{eqnarray}
By $(H_1)$,  they are both the oblique derivative conditions. The
boundary condition on $\Sigma^*=\{(x_1,x_2,x_3)\,:\,
x_1=\max\{\tau,f(x_2,x_3)\}=S$ is
\begin{eqnarray}
\psi=g(x_2,x_3):=(u^+-u^-)(f(x_2,x_3)-\tau)\le0. \label{118}
\end{eqnarray}
Note that there exists $Y\in\bar{\Omega}$ such that $f(Y)=\tau$, so
\begin{eqnarray}\label{zero}
\psi(f(Y), Y)=g(Y)=0.
\end{eqnarray}

We now prove that $\psi\equiv0$  in $D^*$. By \eqref{zero}, it
suffices to show  that $\psi$ is a constant. There are two cases.

\smallskip
\textsc{Case A.}\ $u^+\ge c_1.$ By the strong maximum principle,
$m=\min_{D^*}\psi$ can be achieved only on $\p D^*$ unless $\psi$ is
a constant.

By the Hopf lemma (cf. Lemma 3.4 in \cite{GT}), the minimum $m$ of
$\psi$ can be achieved only on $\bar{S}$ or
$\Gamma\cap\overline{\Sigma_1}$, but not the lateral boundary
$\Gamma$ and the exit $\Sigma_1$ unless $\psi$ is a constant.

\smallskip
(i) Suppose that $m$ is achieved at a point
$P\in\Gamma\cap\overline{\Sigma_1}$. By a locally even reflection
with respect to $\Gamma$ and noting that $\Gamma$ is perpendicular
to $\Sigma_1$ at $P$, $P$ satisfies the interior sphere condition in
the extended neighborhood, as well as $\p_{x_1}\psi\ge0$ due to
\eqref{1160} and \eqref{117}, a contradiction to the Hopf lemma
unless $\psi$ is a constant.

\smallskip
(ii) Suppose that $m$ is achieved on $\overline{S}$. Then $m\le0.$

(a). Let $m=g(X)$ for some $X\in{\Omega}$. Note that, by
\eqref{118}, $\nabla f(X)=0$, so $\nu=(1,0,0)$. By the
Rankine-Hugoniot condition \eqref{rh} and the Bernoulli law
\eqref{102}, as in Lemma \ref{lem103}, we can solve that
\begin{equation}\label{209}
\p_{x_1}\fei(X)=\p_{x_1}\fei_a(X)=u^+.
\end{equation}
Hence,
\begin{equation}\label{210}
\nabla\psi(f(X),X)\cdot \nu(f(X),X)=0.
\end{equation}
By the Hopf lemma, it is impossible unless $\psi$ is
constant.

(b). Let $m=g(X)$ for some $X\in\p\Omega$. Then it is still
necessary to hold $\nabla f(X)=0$ due to the orthogonality of $S$
and $\Gamma$. We also need a locally reflection argument as in (i)
to apply the Hopf lemma to infer that $\psi$ is a constant as in (a)
by \eqref{210}.
%


\medskip \textsc{Case B.}\ $u^+< c_1$. Similar to the analysis in Case A,
now the maximum $M$ of $\psi$ in $D^*$  can be achieved only on
$\bar{S}$ unless $\psi$ is a constant. According to \eqref{118},
$M=0$ and  we may also obtain
$\nabla\psi(f(Y),Y)\cdot\nu(f(Y),Y)=0$, a contradiction to the Hopf
lemma unless $\psi$ is constant.

\medskip
Therefore, $\psi\equiv 0$.  This implies that, for $c_1\ne u^+$,
there is no solution; for $c_1=u^+$, the solution $\fei_\tau$ is
unique (i.e., for any given $\tau\in(-1,1)$, there is only one
transonic shock with its front passing a point in
$\{\tau\}\times\bar{\Omega}$, which is exactly $\fei_\tau$). Since
$\fei_{t+\tau}(x_1,x_2,x_3)=\fei_\tau(x_1-t,x_2,x_3)$, then, for
$c_1=u^+$, the solution to problem \eqref{103}--\eqref{0107} is
unique modulo translation in the $x_1$--direction.

This completes the proof.

\begin{remark}
Note that, in the proof, we need the assumption only that the
downstream flow on the shock-front is subsonic (see \eqref{209} and
\eqref{add2} above). We do not need to assume that the flow behind
the transonic shock-front is always subsonic. In addition, the
assumption $c_1<c_*$ is needed just to guarantee that the transonic
shock-front is restricted in $D$.
\end{remark}

\section{Proof of Theorem \ref{thm102}}

We divide the proof into three steps.

\medskip
{\it Step 1}. Let $\fei$ be a transonic shock solution of problem
\eqref{103}--\eqref{0107} with $D$ replaced by $D'$ and \eqref{106}
replaced by $(H_2)$, as well as $|D^2\fei|$ is bounded in its
subsonic region. Denote its shock-front as
$S=\{(f(x_2,x_3),x_2,x_3)\,:\, (x_2,x_3)\in\Omega\}$. Let
$\tau=\min_{\bar{\Omega}} f$ and $\psi=\fei_\tau-\fei$, where
$\fei_\tau$ is the special transonic shock solution constructed in
Lemma \ref{lem103}. The key point is to show either the maximum or
the minimum of $\psi$ in $D'_*:=\{(x_1,x_2,x_3)\in D'\,:\,
x_1>f(x_2,x_3)\}$ is achieved on $S$. With this, then the rest of
proof is the same as that for Theorem \ref{thm101}. To achieve this,
we now show the following case, Case A, is impossible if $\psi$ is
not a constant.

\medskip
{\it Step 2}. Case A: {\it Both the maximum (might be $\infty$) and
the minimum (might be $-\infty$) are achieved as
$x_1\rightarrow\infty$}.

\medskip
We deduce below that, for this case, there is a contradiction if
$\psi$ is not a constant.


Define $D'_L=\{(x_1,x_2,x_3)\in D'\,:\, f(x_2,x_3)< x_1< L\}$, for
$L$ large, such that $S\subset D'_L$. Without loss of generality, we
assume $M_L=\max_{D'_L}\psi>0, m_L=\min_{D'_L}\psi<0$ (cf.
\eqref{118}), since, for the case $M_L=0$ or $m_L=0$, we can still
apply the Hopf lemma to show $\psi\equiv0$.

Note that $M_L$ is monotonically increasing, while $m_L$ is
monotonically decreasing, as $L$ increases. In this case (Case A),
by assumption, for large $L$, both $M_L$ and $m_L$ are achieved on
$\Sigma_L=\{L\}\times{\bar{\Omega}}$ unless $\psi$ is constant
(which is what we want to prove). Since $\Omega$ is bounded and
$|\nabla\psi|\le |\nabla\fei|+u^+$ is bounded according to $(H_2)$,
we conclude that both $m=\lim_{L\rightarrow\infty}m_L<0$ and
$M=\lim_{L\rightarrow\infty}M_L>0$ are finite. Thus, $\psi$ is a
bounded solution.

Now choose a sequence $\{L_k\}_{k=1}^\infty$ that tends to infinity.
On $\Sigma^{k}:=\Sigma_{L_k}$, we suppose
$M^k:=M_{L_k}=\psi(L_k,X_k)>0,\ m^k:=m_{L_k}=\psi(L_k,Y_k)<0$ for
$X_k, Y_k\in\Omega$.  Since $\fei$ is continuous, there exists
$Z_k\in\bar{\Omega}$ such that $\psi(L_k,Z_k)=0$. 



The following arguments are adopted from \cite{BD}.  Consider
$B_k:=(L_k-2,L_k+2)\times\Omega$ in $\R^3$, by suitable translation,
each $B_k$ ($k\in \mathbb{Z}$) may be transformed onto
$B:=(-2,2)\times\Omega$.

Let $\psi_k(Y)=\psi((L_k,0)+ Y), k=k_0,k_0+1,\cdots$, which is
defined on $B$ and satisfies the linear elliptic equation:
\begin{eqnarray}\label{16}
\sum_{i,j=1}^3a_{ij}^{(k)}(Y)\p_{y_iy_j}\psi_k+\sum_{j=1}^3
b_j^{(k)}(Y)\p_{y_j}\psi_k=0,
\end{eqnarray}
where $a_{ij}^{(k)}(Y)=a_{ij}((L_k,0)+ Y)$ and
$b^{(k)}_i(Y)=b_i((L_k,0)+ Y)$. Obviously, the ellipticity constants
of these equations are the same, and the coefficients are also
uniformly bounded.

Now let $K=(-1,1)\times\Omega$ be a relatively open set of
$B\cup((-2,2)\times\Gamma)$ and  $\bar{K}\subset
B\cup((-2,2)\times\Gamma)$.

Consider $v_k(Y)=M-\psi_k(Y)$, which is positive by definition of
$M$. By \eqref{16}, we have
\begin{eqnarray}
\sum_{i,j=1}^3a_{ij}^{(k)}(Y)\p_{y_iy_j}v_k+\sum_{j=1}^3
b_j^{(k)}(Y)\p_{y_j}v_k=0.
\end{eqnarray}
Applying the boundary Harnack inequality for the oblique derivative
problems (cf. \cite{BCN}, Theorem 2.1, or \cite{P}) to $v_k$, we
have
\begin{eqnarray}
C (M-\psi_k(Y_1))\le C\sup_{\bar{K}}v_k<\inf_{\bar{K}}v_k\le
(M-\psi_k(Y_2))
\end{eqnarray}
for any $Y_1, Y_2\in\bar{K}$, and the positive constant $C$ is
independent of $k$. Taking $Y_1=(0,Z_k)$ and $Y_2=(0,X_k)$, and
letting $k\rightarrow\infty$, we obtain
\begin{eqnarray}
CM\le0,
\end{eqnarray}
which is a contradiction to our assumption that $M>0$.

Therefore, Case A is impossible if $\psi$ is not constant.

\medskip
{\it Step 3}. Suppose that $\psi$ is not a constant. By the strong
maximum principle applied in the domain $D'_*$,  at least one of the
maximum and minimum should be achieved on $\bar{S}$. Applying the
Hopf lemma on $\bar{S}$ as in Section 2, we may also infer
$\psi\equiv 0$. This also contradicts our assumption that $\psi$ is
not a constant.

Therefore, we conclude that $\psi$ is constant. By our choice of
$\tau$, it should be $0$ in $D_*'$. This completes the proof of
Theorem \ref{thm102}.

\section{Proof of Theorem 1.3}

We now show that, for a $C^2$ transonic shock solution $\fei$ to
equations \eqref{101}--\eqref{102}, if $(H_2')$ holds, then
$|D^2\fei|$ is bounded in the unbounded subsonic region. The proof
can be achieved by standard elliptic arguments as sketched below.

{\it Step 1}.\ {\it Decomposition of unbounded domain.} \ For $k\in
\mathcal{Z}$ (positive integers), let
\begin{eqnarray}
D_k=(k-1,k+2)\times\Omega, \qquad
D'_k=(k-\frac12,k+\frac32)\times\Omega.\nonumber
\end{eqnarray}
Clearly, we have $ \dist(D_k, D'_k)=\frac12. $

\medskip
{\it Step 2}. \ {\it Uniformly boundary H\"{o}lder estimate of
gradient and second order derivatives of $\fei$.}\

 For $s=1,2,3,$
let $w=\p_{x_s}\fei$. Then, by differentiating \eqref{101} and
\eqref{102} with respect to $x_s$, we have
\begin{eqnarray}\label{309}
\sum_{i,j}\p_{x_i}(A^{ij}\p_{x_j}w)=0,
\end{eqnarray}
where
\begin{eqnarray}\label{310}
A^{ij}=\rho(|\nabla\fei|^2)\delta_{ij}+2\rho'(|\nabla\fei|^2)\p_{x_i}\fei\p_{x_j}\fei,
\end{eqnarray}
$\delta_{ii}=1$, and $\delta_{ij}=0$ when $i\ne j$ for $i,j=1,2,3$.

By $(H_2')$, this is a uniformly elliptic equation and $A^{ij}$ are
uniformly bounded (independent on $k$).

Now, for large $k>k_0$ such that $D_k$ lies in the subsonic region,
consider equation \eqref{309} in $D'_k$.

For any point $P\in\Gamma\cap\p D'_k$, by  standard  localized
flattening and reflection arguments (here we use the Neumann
boundary condition), in a ball-like neighborhood $B_{3\epsilon}(P)$
of $P$ with the radius $3\epsilon$ depending only on $\Omega$, we
can get by Theorem 8.24 in \cite{GT} that
\begin{eqnarray}
\norm{w}_{C^{\alpha}(\overline{B_{2\epsilon}(P)})}\le
C\norm{w}_{L^2(B_{3\epsilon}(P))}\le C'.
\end{eqnarray}
Note here that $C, C'$, and $\alpha\in(0,1)$ are independent of
$\fei$ and $k$. The second inequality follows from the fact that $w$
is bounded.

Now we analyze equation \eqref{309} whose coefficients, after
extension as above, satisfy
\begin{eqnarray}
\|A^{ij}\|_{C^{\alpha}(B_{2\epsilon}(P))}\le K,
\end{eqnarray}
for a constant $K$ independent of $k>k_0$.  Therefore, by Theorem
8.32 in \cite{GT}, we see
\begin{eqnarray}
\norm{w}_{C^{1,\alpha}(B_\epsilon(P))}\le
C\norm{w}_{C^0(B_{2\epsilon(P)})}\le C''.
\end{eqnarray}
The constants $C$ and $C''$ are independent of $k$, and the second
inequality follows also from the boundedness of $|\nabla\fei|$.

Since $\overline{\Gamma\cap\p D'_k}$ is compact, the number $J$ with
$\overline{\Gamma\cap\p D'_k}\subset \cup_{j=1}^J B_\epsilon(P_j)$
for $P_j\in\Gamma\cap\p D'_k$ is independent of $k$. Then we obtain
a uniform boundary H\"{o}lder estimate of gradient of $w$ with $C$
independent of $k$:
\begin{eqnarray}\label{4.6}
\norm{w}_{C^{1,\alpha}(D_k\cap(\cup_{j=1}^J B_\epsilon(P_j))}\le C.
\end{eqnarray}

\smallskip
{\it Step 3.}\ {\it Uniformly global H\"{o}lder
estimate of gradient and second order derivatives of $\fei$.}\

Choose open sets $D_k^{0}\Subset \hat{D}_k^{0}\Subset D_k$ such that
$D_k-D_k^0\subset U_{j=1}^J B_\epsilon(P_j)$. The elliptic interior
estimate (cf. Theorem 8.24 in \cite{GT}) to equation \eqref{309} in
$\hat{D_k^0}$ tells us that
%
\begin{eqnarray*}
\norm{w}_{C^{\alpha'}(\overline{\hat{D}_k^0})}\le
C\norm{w}_{L^{\infty}(D_k)}
\le C'
\end{eqnarray*}
with $\alpha'\in(0,1)$ and $C,C'>0$ independent of $k$. Without loss
of generality, we assume that $\alpha'\le\alpha$.

Then we see that
$\|A^{ij}\|_{C^{\alpha'}(\overline{\hat{D}_k^0})}\le K$ holds for
$K$ independent of $k$. Now utilizing Theorem 8.32
in \cite{GT} as above, we have
\begin{eqnarray}\label{4.8}
\norm{w}_{C^{1,\alpha'}( \overline{D_k^0})}\le
C\norm{w}_{L^{\infty}(D_k)}
\le C''
\end{eqnarray}
with $C''$ independent of $k$.

Combining \eqref{4.6} with \eqref{4.8}, we conclude $
\|D^2\varphi\|_{C^{\alpha'}(\overline{D_k})}\le \tilde{C} $ with
$\tilde{C}>0$ independent of $k$. This completes the proof of
Theorem 1.3.

\bigskip
{\bf Acknowledgments.} Gui-Qiang Chen's research was supported in
part by the National Science Foundation under Grants DMS-0807551,
DMS-0720925, and DMS-0505473, and the Natural Science Foundation of
China under Grant NSFC-10728101. Hairong Yuan's research was
supported  in part by China Postdoctoral Science Foundation
$(20070410170)$, Shanghai Shuguang Program (07SG29), Fok Ying Tung
Foundation (111002), Shanghai Rising Star Program (08QH14006), and
the National Science Foundation (USA) under Grant DMS-0720925.
The authors thank Mikhail Feldman, Yaguang Wang, and Li Liu for
stimulating discussions and comments.



\begin{thebibliography}{99}

\bibitem{BCN}
Berestycki H., Caffarelli L., and Nirenberg L., Uniform estimates
for regularization of free boundary problems, In: {\it Analysis and
Partial Differential Equations}, Sadosky C. ed., Marcel Decker, pp.
567--619, 1990.

\bibitem{BD}
Bergner, M. and Dittrich, J., A uniqueness and periodic result for
solutions of elliptic equations in unbounded domains, Preprint,
arXiv: 0711.3108v1, 2007 [Available at:
http://arxiv.org/abs/0711.3108]


\bibitem{CKK1}
Cani\'{c}, S., Keyfitz, B.~L., and Kim, E.~H.,
\newblock A free boundary problems for a quasilinear degenerate elliptic
equation: regular reflection of weak shocks,
\newblock {\it Comm. Pure Appl. Math.} {55} (2002), 71--92.

\bibitem{ChF1}
Chen, G.-Q. and Feldman, M., Multidimensional transonic shocks and
free boundary problems for nonlinear equations of mixed type, {\it
J. Amer. Math. Soc.} 16  (2003), 461--494. MR1969202 (2004d:35182)

\bibitem{ChF2}
Chen, G.-Q. and Feldman, M., Steady transonic shocks and free
boundary problems for the Euler equations in infinite cylinders,
{\it Comm. Pure Appl. Math.}  57  (2004), 310--356.  MR2020107
(2004m:35282)

\bibitem{ChF3}
Chen, G.-Q. and Feldman, M., Existence and stability of
multidimensional transonic flows through an infinite nozzle of
arbitrary cross-sections, {\it Arch. Ration. Mech. Anal.}  184
(2007), 185--242. MR2299761

\bibitem{CCF}
Chen, G.-Q., Chen, Jun, and Feldman, M., Transonic shocks and free
boundary problems for the full Euler equations in infinite nozzles,
{\em J. Math. Pures Appl. (9)}  88  (2007), 191--218. MR2348768

\bibitem{CY}
Chen, S. and Yuan, H., Transonic shocks in compressible flow passing
a duct for three-dimensional Euler systems, {\it Arch. Ration. Mech.
Anal.} {187} (2008), 523--556.

\bibitem{CC}
Cole, J.~D. and Cook, L.~P., {\it Transonic Aerodynamics},
North-Holland: Amsterdam, 1986.

\bibitem{CF}
Courant, R. and Friedrichs, K.~O., {\it Supersonic Flow and Shock
Waves}, Interscience Publishers, Inc.: New York, 1948. MR0029615
(10,637c)

\bibitem{EllingLiu}
Elling, V. and Liu, T.-P.,
\newblock  The elliptic principle
for steady and selfsimilar polytropic potential flow, {\it J. Hyper.
Diff. Eqs.} {2} (2005), 909--917.

\bibitem{GT}
Gilbarg, D. and Trudinger, N.~S., {\em Elliptic Partial Differential
Equations of Second Order}, 2nd Edition,
Springer-Verlag: Berlin, 1983. ISBN: 3-540-13025-X, MR0737190
(86c:35035)

\bibitem{Morawetz1}
Morawetz, C.~S.,
\newblock
On the non-existence of continuous transonic flows past profiles
I--III,
\newblock {\it Comm. Pure Appl. Math.}
{\bf 9} (1956), 45--68; {\bf 10} (1957), 107--131; {\bf 11} (1958),
129--144.


\bibitem{P}
Pinchover Y., On positive Liouville theorems and asymptotic behavior
of solutions of Fuchsian type elliptic operators, {\it Ann. Inst.
Henri Poincar\`e: Nonlinear Analysis} {11} (1994), 313--341.

\bibitem{XY}
Xin, P. and Yin, H., Transonic shock in a nozzle. I. Two-dimensional
case, {\it Comm. Pure Appl. Math.}  58 (2005), 999--1050. MR2143525
(2006c:76079)

\bibitem{Yu1}
Yuan H.,  A remark on determination of transonic shocks in divergent
nozzles for steady compressible Euler flows, {\it Nonlinear
Analysis: Real World Appl.} 9 (2008), 316--325.

\bibitem{Yu2}
Yuan, H., On transonic shocks in two-dimensional variable-area ducts
for steady Euler system, {\it SIAM J. Math. Anal.}  38 (2006),
1343--1370. MR2274487


\bibitem{Yu3}
Yuan, H., Transonic shocks for steady Euler flows with cylindrical
symmetry, {\it Nonlinear Anal.}  66 (2007), 1853--1878. MR2307052


\end{thebibliography}
\end{document}